\documentclass{amsart}
\usepackage{shortsalch, xy}
\xyoption{all}
\title{Failure of flat descent of model structures on module categories.}
\date{March 2014.}
\begin{document}
\begin{abstract}
We prove that the collection of model structures on (quasicoherent) module categories does not obey flat descent. In particular, it fails to be a separated presheaf, in the fppf topology, on Artin stacks.
\end{abstract}
\maketitle
\tableofcontents
\section{Introduction.}

Recently H. Bacard asked us a question we had ourselves already been wondering about: does there exist a ``local-to-global principle'' for model structures on module categories? In other words, if one has an open cover of a scheme or a stack, and the structure of a model category on the (quasicoherent) module category of each open in the cover, can one ``paste'' the model structures into a model structure on (quasicoherent) modules over the whole scheme or stack? In this note we show that the answer to this question is {\em no}, at least in the flat topology. 

We now provide more details. For $X$ a scheme or an algebraic stack, let $Model(X)$ be the collection of cofibrantly generated closed model structures on the category of quasicoherent $\mathcal{O}_X$-modules.
(In general this collection is not known to form a set, or even a class; it is just a collection. For many practical choices of $X$, however, it forms a set and
even a finite set. See \cite{computationofmodelstructures} for details.) 

One wants to know if $Model$ is a sheaf in various topologies. More precisely: 
let $X$ be an algebraic stack in a topology $\tau$ and choose an $\tau$-cover $Y \rightarrow X$. One wants to know if 
\[ \xymatrix{ Model(X) \ar[r] &  Model(Y) \ar@<1ex>[r]\ar@<-1ex>[r] & Model(Y\times_X Y) }\]
is an equalizer sequence. In other words, given a model structure on the category of $\mathcal{O}_Y$-modules which restricts compatibly 
on overlaps between components of the cover $Y$, does there exist a unique model structure on the category of $\mathcal{O}_X$-modules
which restricts to the chosen one on $\mathcal{O}_Y$-modules? This is the problem of {\em descent of model structures on module categories.}

In this short note we give a negative answer to this question for the fppf topology, by producing an explicit counterexample in Thm.~\ref{main thm}.
We do not know an answer to the question for coarser topologies but we (with H. Bacard, M. Frankland, D. Schaeppi) have explored the same question
for the Zariski topology, where it seems more plausible that this question of descent has a positive answer, and perhaps interesting results can be found in that direction. We are grateful to J. F. Jardine for hosting us during
a visit to Western University, where these questions came up in conversation.

\section{$Model$ as a presheaf.}

\begin{definition}\label{def of model as presheaf}
Let $\tau$ be one of the following topologies: Zariski, \'{e}tale, or fppf.
Let $X$ be an algebraic stack in the topology $\tau$ and let $\Alg\Stacks_{\tau}(X)$ be the category of algebraic stacks $Y$ in the $\tau$-topology equipped with
maps $Y\rightarrow X$ which are components of some covering family in $\tau$.

For every object $Y$ of $\Alg\Stacks_{\tau}(X)$, let $Model(Y)$ be the collection of cofibrantly generated closed model structures on the category of quasicoherent $\mathcal{O}_Y$-modules.

For every map $f: Y^{\prime}\rightarrow Y$ in $\Alg\Stacks_{\tau}(X)$, let $Model(f): Model(Y) \rightarrow Model(Y^{\prime})$ be the partially-defined function sending
a model structure on quasicoherent $\mathcal{O}_{Y}$-modules to the transfer model structure, if it exists, on quasicoherent $\mathcal{O}_{Y^{\prime}}$-modules, i.e., the model structure
in which a map $g: M \rightarrow N$ of quasicoherent $\mathcal{O}_{Y^{\prime}}$-modules is a weak equivalence (respectively, fibration) if and only if
$f_*g: f_*M \rightarrow f_*N$ is a weak equivalence (respectively, fibration)
in the given model structure on quasicoherent $\mathcal{O}_Y$-modules.
\end{definition}
In the literature there is a small point of difference between various definitions of algebraic stack: the issue is what condition one requires of the diagonal map. The reader can choose whatever diagonal conditions they like best, for the purposes of Def.~\ref{def of model as presheaf}. Our counterexample (in Thm.~\ref{main thm}) to fppf descent of model structures, however,
works in all the definitions of fppf algebraic stack which are in common circulation, since in our counterexample all diagonal maps are affine, hence also quasicompact, quasiseparated, etc. 

The transfer model structure exists if and only if, for every 
morphism $g$ of quasicoherent $\mathcal{O}_{Y^{\prime}}$-modules which is a 
transfinite composite of pushouts of coproducts of morphisms 
$f^*h$ with $h$ an acyclic cofibration in the given model structure
on $\mathcal{O}_Y$-modules, the morphism $f_*g$ is a weak equivalence
in the given model structure on $\mathcal{O}_Y$-modules. This is 
a simple application of Crans's theorem on existence of transfer
model structures, from \cite{MR1346427}.

The assumption that our model structures be cofibrantly generated
is an unnecessary assumption in many situations of practical interest.
In \cite{computationofmodelstructures} we show that, if a category admits a 
``saturating basis,'' then every model structure on that category
is cofibrantly generated. It then follows from the Cohen-Kaplansky theorem (from \cite{MR0043073}) that, if $R$ is an Artinian principal ideal ring, 
then every model structure on the category of
$R$-modules is cofibrantly generated. We give more information and 
explicit computations in \cite{computationofmodelstructures}. It is probably the case that
very many module categories have the property that all their model structures
are cofibrantly generated. 
In any case, leaving out the adjectives ``cofibrantly generated'' or ``closed,'' or both, from Def.~\ref{def of model as presheaf} does not deter Thm.~\ref{main thm} from providing a counterexample to $Model$ being a sheaf. 

\section{$Model$ is not an fppf sheaf.}

\begin{definition}
Let $k$ be a field of characteristic $2$ and let $\alpha_2$ be the algebraic group over $k$ 
co-represented by the Hopf algebra $k[x]/x^2$ with $x$ primitive. 
That is, $\alpha_2$ is the algebraic group given on affines by letting $\alpha_2(\Spec R)$ be the group of square-zero elements of
$R$, under addition.

Let $B\alpha_2$ be the (Artin) stack of $\alpha_2$-torsors, i.e., the
fppf algebraic stack associated to the group(oid) affine scheme
$(\Spec k, \Spec k[x]/x^2)$. We will write $\Spec k \rightarrow B\alpha_2$
for the cover naturally associated to the presentation of
$B\alpha_2$ given by the groupoid affine scheme $(\Spec k, \Spec k[x]/x^2)$.
Finally, let $j_1,j_2$ denote the two (partially-defined) maps
\[ \xymatrix{  Model(\Spec k) \ar@<1ex>[r]\ar@<-1ex>[r] & Model(\Spec k \times_{B\alpha_2}\Spec k) }\]
and let $i$ denote the (partially-defined) map
\[ Model(B\alpha_2)\rightarrow Model(\Spec k).\]
\end{definition}

\begin{theorem}\label{main thm}
There exist two distinct elements $a,b$ in $Model(B\alpha_2)$ 
such that:
\begin{itemize}
\item $i(a)$ and $i(b)$ are well-defined,
\item $i(a) = i(b)$,
\item and $j_1(i(a))$ and $j_2(i(a))$ are well-defined (and consequently
$j_1(i(a)) = j_2(i(a))$.
\end{itemize}
Consequently, even if $Model$ is a well-defined presheaf, it is not a 
separated presheaf in the fppf topology, hence not a sheaf in the fppf topology.
\end{theorem}
\begin{proof}
We first define the model structures $a,b$ on quasicoherent $\mathcal{O}_{B\alpha_2}$-modules. By the usual descent argument, quasicoherent $\mathcal{O}_{B\alpha_2}$-modules are equivalent to $k[x]/x^2$-comodules. 

First, a definition: suppose $M$ is a right $k[x]/x^2$-comodule with structure map
$\psi: M \rightarrow M\otimes_k k[x]/x^2$.
We say that {\em $M$ is $x$-trivial} if the composite map
\[ M \stackrel{(\id_M\otimes_k\eta) - \psi}{\longrightarrow} M\otimes_k k[x]/x^2
\stackrel{\id_M\otimes_k \epsilon}{\longrightarrow} M\otimes_k k\]
is zero, where $\eta$ is the unit map $\eta: k \rightarrow k[x]/x^2$.
In other words, $M$ is $x$-trivial if and only if, for all $m\in M$, 
$\psi(m)$ has no nonzero terms involving $x$.

Note that every right $k[x]/x^2$-comodule $M$ has a maximal $x$-trivial 
subcomodule 
$triv(M)$, and that the inclusion 
$triv(M)\rightarrow M$ expresses the $x$-trivial comodules
as a coreflective subcategory of the category of
$k[x]/x^2$-comodules.

Let model structure $a$ on the category of $k[x]/x^2$-comodules be 
the model structure in which:
\begin{itemize}
\item the cofibrations are the pushouts of maps between $x$-trivial comodules, 
\item the weak equivalences are the maps inducing an isomorphism on
the maximal $x$-trivial subcomodules,
\item and all maps are fibrations.
\end{itemize}
In \cite{MR779198} (also, more directly, in our \cite{fibrantreplacementmonads})
the necessary results on factorization systems are proven to ensure that, 
if $\mathcal{C}$ is an
category with finite limits and colimits and $\mathcal{A}$ is a coreflective
replete subcategory of $\mathcal{C}$, then there exists a model structure on
$\mathcal{C}$ in which the cofibrations are the pushouts of maps
between objects in $\mathcal{A}$, the weak equivalences are the maps
inducing an isomorphism on applying the coreflector functor, and all maps
are fibrations. Our model structure $a$ is the special case
of this theorem where $\mathcal{A}$ consists of the $x$-trivial
comodules.

Let model structure $b$ on the category of $k[x]/x^2$-comodules be 
the discrete model structure, i.e., the model structure in which:
\begin{itemize}
\item all maps are cofibrations,
\item the weak equivalences are the isomorphisms,
\item and all maps are fibrations.
\end{itemize}

We claim that $i(a)$ and $i(b)$ both exist and in fact are each equal to the
discrete model structure on the category of $k$-modules.
Indeed, under the identification of quasicoherent $\mathcal{O}_{B\alpha_2}$-modules with $k[x]/x^2$-comodules, we have the standard identification of the
direct image and inverse image functors in terms of comodules:
the map of stacks 
$\Spec k\stackrel{f}{\longrightarrow} B\alpha_2$ induces the direct image 
functor
\[ f_*: \Mod(k) \rightarrow \Comod(k[x]/x^2)\]
sending a $k$-module $M$ to the extended comodule $M \otimes_k k[x]/x^2$,
and the inverse image functor
\[ f^*: \Comod(k[x]/x^2) \rightarrow \Mod(k)\]
sending a $k[x]/x^2$-comodule to its underlying $k$-module, i.e., forgetting the
coaction map. 

So in the transfer model structure $i(a)$ on $k$-modules,
a morphism $g$ is a weak equivalence if and only if
$g\otimes_k k[x]/x^2$ is a weak equivalence in $a$. But this implies $g$
must be an isomorphism, since the map induced by $g\otimes_k k[x]/x^2$ on the 
maximal $x$-trivial subcomodules is $g$ itself.
Hence the weak equivalences in $i(a)$ are the same as the discrete
weak equivalences. The same is trivially true for fibrations.

Since $b$ is the discrete model structure on $k[x]/x^2$-comodules, 
it is trivially true that
the weak equivalences and fibrations in $i(b)$ agree with those of the
discrete model structure on $k$-modules.

Consequently $i(a)$ and $i(b)$ both exist and are equal to the discrete
model structure on $k$-modules.

Now all that remains is to show that $j_1(i(a))$ and $j_2(i(a))$ both exist.
That they are equal if they exist 
is trivial (for example, because $j_1 =j_2$ in this setting,
since $k[x]/x^2$ is a Hopf algebra and not only a Hopf algebroid!).
The model structure $j_1(i(a))$ is the discrete model structure on $k$-modules
transferred to quasicoherent $\mathcal{O}_{\Spec k\times_{B\alpha_2}\Spec k} \cong \mathcal{O}_{\Spec k[x]/x^2}$-modules, i.e., $k[x]/x^2$-modules, along the 
direct image map, i.e., restriction of scalars. So a morphism $g$ of
$k[x]/x^2$-modules is a weak equivalence in $j_1(i(a))$ if and only if
its underlying map of $k$-modules is a weak equivalence in $i(a)$, i.e.,
if and only if $g$ is an isomorphism.
Similarly, a morphism $g$ of
$k[x]/x^2$-modules is a fibration in $j_1(i(a))$ if and only if
its underlying map of $k$-modules is a fibration in $i(a)$, i.e.,
all morphisms $g$ of $k[x]/x^2$-modules are fibrations.
Hence the weak equivalences and the fibrations in $j_1(i(a))$ 
each agree with those in the discrete model structure on $k[x]/x^2$-modules.
Hence $j_1(i(a)) = j_2(i(a)) = j_1(i(b)) = j_2(i(b))$ exists and is the
discrete model structure on $k[x]/x^2$-modules.
\end{proof}

\bibliography{/home/asalch/texmf/tex/salch}{}
\bibliographystyle{plain}
\end{document}